\documentclass[10pt,twoside]{article}
\usepackage{amsmath,amsbsy,amsfonts}
\usepackage{graphicx,color}
\usepackage{upgreek}
\everymath{\displaystyle}
\newtheorem{theorem}{Theorem}[section]
\newtheorem{lemma}[theorem]{Lemma}
\newtheorem{proposition}[theorem]{Proposition}
\newtheorem{corollary}[theorem]{Corollary}

\newtheorem{claim}[theorem]{Claim}

\newenvironment{proof}[1][Proof]{\noindent\textbf{#1.} }{\ \rule{0.5em}{0.5em}}

\newcommand{\R}{\mathbb{R}}
\newcommand{\Z}{\mathbb{Z}}
\newcommand{\N}{\mathbb{N}}
\newcommand{\M}{\mathcal{M}}

\begin{document}

\title{\bf Existence of ground state solution and concentration of maxima for a class of indefinite variational problems}

\author{Claudianor O. Alves\thanks{C. O. Alves was partially supported by CNPq/Brazil 304804/2017-7 and INCT-MAT, coalves@mat.ufcg.edu.br}\, , \, \ Geilson F. Germano \thanks{G. F. Germano was partially supported by CAPES, geilsongermano@hotmail.com}\vspace{2mm}
	\and {\small  Universidade Federal de Campina Grande} \\ {\small Unidade Acad\^emica de Matem\'{a}tica} \\ {\small CEP: 58429-900, Campina Grande - Pb, Brazil}\\}

\date{}
\maketitle

\begin{abstract} 
In this paper we study the existence of ground state solution  and concentration of maxima for a class of  strongly indefinite problem like
$$
\left\{\begin{array}{l}
-\Delta u+V(x)u=A(\epsilon x)f(u) \quad \mbox{in} \quad  \R^{N}, \\
u\in H^{1}(\R^{N}),
\end{array}\right. \eqno{(P)_{\epsilon}}
$$
where $N \geq 1$, $\epsilon$ is a positive parameter, $f: \mathbb{R} \to \mathbb{R}$ is a continuous function with subcritical growth and $V,A: \mathbb{R}^{N} \to \mathbb{R}$ are continuous functions verifying some technical conditions. Here $V$ is a $\mathbb{Z}^N$-periodic function, $0 \not\in \sigma(-\Delta + V)$, the spectrum of $-\Delta +V$,  and  
$$
0 < \inf_{x \in \R^{N}}A(x)\leq \displaystyle\lim_{|x|\rightarrow+\infty}A(x)<\sup_{x \in \R^{N}}A(x). 
$$

	\vspace{0.3cm}
	
	
	\vspace{0.3cm}
	
\end{abstract}

\section{Introduction}

This paper concerns with the existence of ground state solution  and concentration of maxima for the following semilinear Schr\"odinger equation
$$
\left\{\begin{array}{l}
-\Delta u+V(x)u=A(\epsilon x)f(u) \quad \mbox{in} \quad  \R^{N}, \\
u\in H^{1}(\R^{N}),
\end{array}\right. \eqno{(P)_{\epsilon}}
$$
where $N \geq 1$, $\epsilon$ is a positive parameter, $f: \mathbb{R} \to \mathbb{R}$ is a continuous function with subcritical growth and $V,A: \mathbb{R} \to \mathbb{R}$ are continuous functions verifying some technical conditions. 

In whole this paper, $V$ is $\mathbb{Z}^N$-periodic with   
$$
0 \not\in \sigma(-\Delta + V), \quad \mbox{the spectrum of } \quad -\Delta +V,  \eqno{(V_1)}
$$
which becomes the problem strongly indefinite. Related to the function $A$, we assume that it is a continuous function satisfying 
$$
A(0)=\max_{x \in \mathbb{R}^N}A(x) \quad \mbox{and} \quad 0< A_0=\inf_{x \in \R^{N}}A(x)\leq \displaystyle\lim_{|x|\rightarrow+\infty}A(x)<\sup_{x \in \R^{N}}A(x). \eqno{(A_1)}
$$		

The present article has as first motivation some recent articles that have studied the existence of solution for related problems with $(P)_{\epsilon}$, more precisely for strongly indefinite problems of the type  
$$
\left\{\begin{array}{l}
-\Delta u+V(x)u=f(x,u), \quad \mbox{in} \quad  \R^{N}, \\
u\in H^{1}(\R^{N}).
\end{array}\right. \eqno{(P_1)}
$$
In \cite{Szulkin1}, Kryszewski and Szulkin  have studied the existence of solution for $(P_1)$ by supposing the condition $(V_1)$. Related to the function $f:\mathbb{R}^N \times \mathbb{R} \to \mathbb{R}$, they assumed that $f$ is continuous, $\mathbb{Z}^N$-periodic in $x$ with
$$
|f(x,t)| \leq c(|t|^{q-1}+|t|^{p-1}), \quad \forall t \in \mathbb{R} \quad \mbox{and} \quad x \in \mathbb{R}^N \eqno{(h_1)}
$$ 
and
$$
0<\alpha F(x,t) \leq tf(x,t) \quad \forall (x,t) \in \mathbb{R}^N \times \mathbb{R}  \setminus \{0\} , \quad F(x,t)=\int_{0}^{t}f(x,s)\,ds \eqno{(h_2)}
$$
for some $c>0$, $\alpha >2$ and $2<q<p<2^{*}$ where $2^{*}=\frac{2N}{N-2}$ if $N \geq 3$ and $2^{*}=+\infty$ if $N=1,2$. The above hypotheses guarantee that the energy functional associated with $(P_1)$  given by 
$$
J(u)=\frac{1}{2}\int_{\mathbb{R}^N}(|\nabla u|^{2}+V(x)|u|^{2}\,dx)-\int_{\mathbb{R}^N}F(x,u)\,dx, \,\, \forall u \in H^{1}(\R^N),
$$
is well defined and belongs to $C^{1}(H^{1}(\mathbb{R}^N), \mathbb{R})$. From $(V_1)$, there is an equivalent inner product $\langle \;\; , \;\; \rangle $  in $H^{1}(\mathbb{R}^N)$ such that 
$$
J(u)=\frac{1}{2}\|u^+\|^{2}-\frac{1}{2}\|u^-\|^{2} -\int_{\mathbb{R}^N}F(x,u)\,dx,
$$
where $\|u\|=\sqrt{\langle u,u \rangle}$ and  $H^{1}(\mathbb{R}^N) = E^{+} \oplus E^-$ corresponds to the spectral decomposition of $- \Delta + V $ with respect to the
positive and negative part of the spectrum with $u = u^{+}+u^{-}$, where $u^{+} \in E^{+}$ and $u^{-} \in E^{-}$.  In order to show the existence of solution for $(P_1)$, Kryszewski and Szulkin  introduced a new and interesting generalized link theorem. In \cite{LiSzulkin},  Li and Szulkin have improved the generalized link theorem obtained in \cite{Szulkin1} to prove the existence of solution for a class of strongly indefinite problem with $f$ being asymptotically linear at infinity.   

The link theorems above mentioned have been used in a lot of papers, we would like to cite   Chabrowski and Szulkin \cite{CS}, do \'O and Ruf \cite{DORUF},   Furtado and Marchi \cite{furtado}, Tang \cite{tang, tang2} and their references.

Pankov and Pfl\"uger \cite{Pankov-Pfluger} also have considered the existence of solution for problem $(P_1)$ with the same conditions considered in \cite{Szulkin1}, however the approach is based on an approximation technique of periodic function together with the linking theorem due to Rabinowitz \cite{Rabinowitz}. Later, Pankov \cite{Pankov} has studied the  existence of solution for problems of the type  
$$
\left\{\begin{array}{l}
-\Delta u+V(x)u=\pm f(x,u), \quad \mbox{in} \quad  \R^{N}, \\
u\in H^{1}(\R^{N}),
\end{array}\right. \eqno{(P_2)}
$$
by supposing $(V_1)$, $(h_1)-(h_2)$ and employing the same approach explored in \cite{Pankov-Pfluger}.  In   \cite{Pankov} and \cite{Pankov-Pfluger}, the existence of ground state solution has been established by supposing $f \in C^{1}(\mathbb{R}^N,\mathbb{R})$ and that there is $\theta \in (0,1)$ such that
$$
0<t^{-1}f(x,t)\leq \theta f'_t(x,t), \quad \forall t \not=0 \quad \mbox{and} \quad x \in \mathbb{R}^N. \eqno({h_3})
$$
In  \cite{Pankov}, Pankov found a ground state solution by minimizing the energy functional  $J$ on the set
$$
\mathcal{O}=\left\{u\in H^{1}(\R^{N})\setminus E^{-}\ ;\ J'(u)u=0\text{ and }J'(u)v=0,\forall\ v\in E^{-}\right\}. 
$$ 
The reader is invited to see that if $J$ is strongly definite, that is, when $E^{-}=\{0\}$,  the set $\mathcal{O}$ is exactly the Nehari manifold associated with $J$. Hereafter, we say that  $u_0 \in H^{1}(\mathbb{R}^{N})$ is a {\it ground state solution} if 
$$
J'(u_0)=0 \quad \mbox{and} \quad J(u_0)=\inf_{w \in \mathcal{O}}J(w).
$$

In \cite{SW}, Szulkin and Weth established the existence of ground state solution for problem $(P_1)$ by completing the study made in  \cite{Pankov}, in the sense that, they also minimize the energy functional on  $\mathcal{O}$, however they have used more weaker conditions on $f$, for example $f$ is continuous, $\mathbb{Z}^N$-periodic in $x$ and satisfies 
$$
|f(x,t)|\leq C(1+|t|^{p-1}), \;\; \forall t \in \mathbb{R} \quad \mbox{and} \quad x \in \mathbb{R}^N  \eqno{(h_4)}
$$ 
for some $C>0$ and $p \in (2,2^{*})$.
$$
f(x,t)=o(t) \,\,\, \mbox{uniformly in } \,\, x \,\, \mbox{as} \,\, |t| \to 0. \eqno{(h_5)}
$$
$$
F(x,t)/|t|^{2} \to +\infty \,\,\, \mbox{uniformly in } \,\, x \,\, \mbox{as} \,\, |t| \to +\infty, \eqno{(h_6)}
$$
and
$$
t\mapsto f(x,t)/|t| \,\,\, \mbox{is strictly increasing on} \,\,\, \mathbb{R} \setminus \{0\}. \eqno{(h_7)}
$$
The same approach was used by Zhang, Xu and Zhang \cite{ZXZ, ZXZ2} to study a class of indefinite and asymptotically periodic problem.

After a bibliography review, we have observed that there are no papers involving strongly indefinite problem whose the nonlinearity is of the form  
$$
f(x,t)=A(\epsilon x)f(t), \quad \forall x \in \mathbb{R}^N \quad \mbox{and} \quad \forall t \in \mathbb{R},
$$
with $A$ satisfying $(A_1)$ and $\epsilon >0$. The motivation to consider this type of nonlinearity comes from many studies involving the existence and concentration of standing-wave solutions for the nonlinear Schr\"{o}dinger
equation
$$
ih \displaystyle \frac{\partial \Psi}{\partial t}=-h^{2}\Delta
\Psi+(V(x)+E)\Psi-g(x,\Psi)\,\,\, \mbox{for all}\,\,\, x \in
\mathbb{R}^{N},\eqno{(NLS)}
$$
where $N \geq 1$, $h > 0$ is a parameter and $V,g$ are continuous functions verifying some conditions. This class of equation is one of the main objects of the quantum physics, because it appears in problems that involve nonlinear optics, plasma physics and condensed matter physics.

Knowledge of the solutions for the elliptic equation like
$$
\ \  \left\{
\begin{array}{l}
- h^{2} \Delta{u} + V(x)u=g(x,u)
\ \ \mbox{in} \ \ \mathbb{R}^{N},
\\
u \in H^{1}(\mathbb{R}^{N}),
\end{array}
\right.
\eqno{(S)_{h}}
$$
or equivalently
$$
\ \  \left\{
\begin{array}{l}
-\Delta{u} + V(h x)u=g(h x, u)
\ \ \mbox{in} \ \ \mathbb{R}^{N},
\\
u \in H^{1}(\mathbb{R}^{N}),
\end{array}
\right.
\eqno{(S')_{h}},
$$
has a great importance in the study of standing-wave solutions
of $(NLS)$, which is a solution of the form $\Psi(x,t)=e^{-itE/h}u(x)$. In recent years, the existence and concentration of
positive solutions for general semilinear elliptic equations
$(S)_h$ have been extensively studied, see for example, Floer and Weinstein \cite{FW}, Oh
\cite{O1,O2}, Rabinowitz \cite{rabinowitz}, Wang \cite{WX}, Ambrosetti and Malchiodi \cite{AM}, Ambrosetti, Badiale and Cingolani  \cite{ABC},  del Pino and Felmer \cite{DF1}  and their references.

In some of the above mentioned papers, the existence,  multiplicity and concentration of positive solutions have been obtained in connection with the geometry of the potential $V$ by supposing that  
$$
\inf(\sigma(-\Delta + V))>0.
$$
By using the above condition, in some cases, it is possible to prove that the energy functional satisfies the mountain pass geometry, and that the mountain pass level is a critical level. In some papers it was proved that the maximum points of the solutions are close to the set
$$
\mathcal{V} = \left\{x \in \mathbb{R}^{N} \ : \ V(x) = \min_{z \in \mathbb{R}^{N}}V(z) \right\},
$$
when $h$ is small enough. Moreover, in a lot of articles, the multiplicity of solutions is associated with the topology richness of $\mathcal{V}$.

In \cite{rabinowitz}, by a mountain pass argument,  Rabinowitz proved the
existence of positive solutions of $(S)_{h}$, for $h > 0$ small, with $g(x,t)=g(t)$ whenever
$$
\liminf_{|x| \rightarrow \infty} V(x) > \inf_{x \in
	\mathbb{R}^N}V(x)=V_{0} >0. 
$$
Later Wang \cite{WX} showed that these solutions concentrate at global minimum points of $V$ as  $h$ tends to 0.

In \cite{DF1}, del Pino and Felmer have found solutions that concentrate around local minimum of $V$ by introducing of a penalization method.
More precisely, they assume that
$$
V(x)\geq \inf_{z \in \mathbb{R}^N}V(z)=V_{0}>0 \,\,\, \mbox{for all}\,\, \ x \in \mathbb{R}^{N}
$$
and the existence of an open and bounded set $\Omega \subset \mathbb{R}^{N}$ such that
$$
\inf_{x \in \Omega}V(x)< \min_{x \in
	\partial \Omega}V(x). 
$$

Here, we intend to study the existence of standing-wave solutions for $(NSL)$ by supposing $h=1$  and $g$ be a function of the type
$$
g(x,t)=A(\epsilon x)f(t),
$$
where $\epsilon$ is a positive number with $V,A$ satisfying the conditions $(V_1)$ and $(A_1)$ respectively. More precisely, we will prove the existence of ground state solution $u_\epsilon$ for $(P)_\epsilon$ when $\epsilon$ is small enough. After, we study the concentration of the maximum points of $|u_\epsilon|$ with related to the set of maximum points of $A$. We would like point out that one of the main difficulties is the loss of the mountain pass geometry, because we are working with a strongly indefinite problem. Then, if $I_\epsilon$ denotes the  energy functional associated with $(P)_\epsilon$, we were taken to do a careful study involving the behavior of number $c_{\epsilon}$ given by 
\begin{equation} \label{cepsilon}
c_\epsilon=\inf_{u \in \M_{\epsilon}}I_\epsilon(u)
\end{equation}
where
\begin{equation} \label{Mepsilon}
\M_{\epsilon}=\left\{u\in H^{1}(\R^{N})\setminus E^{-}\ ;\ I_{\epsilon}'(u)u=0\text{ and }I_{\epsilon}'(u)v=0,\forall\ v\in E^{-}\right\}. 
\end{equation}
The understanding of the behavior of $c_\epsilon$ is a key point in our approach to show the existence of ground state solution  and concentration of maxima when $\epsilon$ is small enough.

Hereafter, $f:\R \rightarrow\R$ is a continuous function that verifies the following assumptions: \\

\noindent $(f_1)$ \, $\displaystyle\frac{f(t)}{t}\rightarrow0$ as $t\rightarrow0$. \\

\noindent $(f_2)$ \, $\displaystyle\limsup_{|t|\rightarrow +\infty}\frac{|f(t)|}{|t|^{q}}<+\infty$ for some $q\in(1,2^{*}-1)$.

\noindent $(f_3)$ $t\mapsto {f(t)}/{t}$ is increasing on $(0,+\infty)$ and decreasing on $(-\infty,0)$.

\noindent $(f_4)$\, (Ambrosetti-Rabinowitz) There exists $\theta>2$ such that 
$$
0<\theta F(t)\leq f(t)t,\ \forall\ t\neq0
$$
where $F(t):=\int_{0}^{t}f(s)ds$. \\
	
Our main theorem is the following 
\begin{theorem} 
Suppose that  $(V_1), (A_1)$ and $(f_1)-(f_4)$ hold. Then,
there exists  $\epsilon_{0}> 0$ such that $(P)_{\epsilon}$ has a ground state solution $u_\epsilon$ for all $\epsilon \in (0,\epsilon_0)$. Moreover, if $x_{\epsilon} \in \mathbb{R}^{N}$ denotes a global maximum point of $|u_{\epsilon}|$, then
\[
\lim_{\epsilon \rightarrow 0}A(\epsilon x_{\epsilon}) = \sup_{x \in \mathbb{R}^N}A(x).
\]
\end{theorem}

The plan of the paper is as follows: In Section 2 we  make a study involving the autonomous problem. In Section 3 it is showed the existence of ground state solution for $\epsilon$ small, while in Section 4 we establish the concentration of maxima.

\vspace{0.5 cm}

\noindent\textbf{Notation.} In this paper, we use the following notations:

\begin{itemize}
	\item $o_{n}(1)$ denotes a sequence that converges to zero.
	\item If $g$ is a mensurable function, the integral $\int_{\mathbb{R}^N}g(x)\,dx$ will be denoted by $\int g(x)\,dx$.	
	\item $B_{R}(z)$ denotes the open ball with center  $z$ and
	radius $R$ in $\mathbb{R}^N$.
	\item  The usual norms in $H^{1}(\R^N)$ and $L^{p}(\R^N)$ will be denoted by
		$\|\;\;\;\|_{H^{1}(\R^N)}$ and $|\;\;\;|_{p}$ respectively. 
	\item For each $u \in H^{1}(\R^N)$, the equality $u=u^+ + u^-$ yields $u^+ \in E^+$ and $u^- \in E^-$.	
\end{itemize}

\section{Some results involving the autonomous problem.}

Consider the following autonomous problem 
$$
\left\{\begin{array}{l}
-\Delta u+V(x)u=\lambda f(u)\quad \mbox{in} \quad  \R^{N}, \\ 
u\in H^{1}(\R^{N}),
\end{array}\right.\eqno{(AP)_{\lambda}}
$$
where $\lambda >0$ and $V, f$ verify the conditions $(V_1)$ and $(f_1)-(f_4)$ respectively.   Associated with $(AP)_{\lambda}$ we have the energy functional $J_\lambda:H^{1}(\mathbb{R}^N) \to \mathbb{R}$ given by 
$$
J_\lambda(u)=\frac{1}{2}\int (|\nabla u|^{2}+V(x)|u|^{2}\,dx)-\lambda \int F(u)\,dx, 
$$
or equivalently   
$$
J_\lambda(u)=\frac{1}{2}\|u^+\|^{2}-\frac{1}{2}\|u^-\|^{2} -\lambda \int F(u)\,dx.
$$
In what follows, let us denote by $d_\lambda$ the real number defined by
\begin{equation} \label{dlambda}
d_\lambda=\inf_{u \in \mathcal{N}_{\lambda}}J_\lambda(u);
\end{equation}
where
\begin{equation} \label{Nlambda}
\mathcal{N}_{\lambda}=\left\{u\in H^{1}(\R^{N})\setminus E^{-}\ ;\ J_{\lambda}'(u)u=0\text{ and }J_{\lambda}'(u)v=0,\forall\ v\in E^{-}\right\}. 
\end{equation}
Moreover, for each $u \in H^{1}(\R^N)$, the sets $E(u)$ and $\hat{E}(u)$ designate  
\begin{equation} \label{conjuntoE}
E(u)=E^{-}\oplus \R u\ \text{ and }\ \hat{E}(u)=E^{-}\oplus [0,+\infty)u.
\end{equation}

The reader is invited to observe that $E(u)$ and $\hat{E}(u)$ are independent of $\lambda$, more precisely they depend on only of the operator $-\Delta + V$. This remark is very important because these sets will be used in the next sections.

In \cite{SW}, Szulkin and Weth proved that for each $\lambda>0$, problem $(AP)_\lambda$ possesses a ground state solution $u_\lambda \in H^{1}(\R^N)$, that is, 
$$
u_\lambda \in \mathcal{N}_{\lambda}, \quad J_\lambda(u_\lambda)=d_\lambda \quad \mbox{and} \quad J_{\lambda}'(u)=0.
$$
In the above mentioned paper, the authors also proved that
\begin{equation} \label{positivo}
0<d_\lambda=\inf_{u\in E^{+}\setminus\{0\}}\max_{v\in\widehat{E}(u)}J_{\lambda}(u).
\end{equation}
Moreover, an interesting and important fact is that for each $u\in H^{1}(\R^{N})\setminus E^{-}$, $\mathcal{N}_\lambda\cap\hat{E}(u)$ is a singleton set and the element of this set is the unique global maximum of $J_{\lambda}|_{\hat{E}(u)}$, that is, there are $t^* \geq 0$ and $v^* \in E^{-}$ such that
\begin{equation} \label{maximo}
J_{\lambda}(t^*u+v^*)=\displaystyle\max_{w \in \widehat{E}(u)}J_{\lambda}(w).
\end{equation}

The next two lemmas will be used in the study of the behavior of $d_\lambda$ and $c_\epsilon$.

\begin{lemma}\label{translacao}
	For all $u=u^+ + u^-\in H^{1}(\R^{N})$ and $y\in\Z^{N}$, if $u_{y}(x):=u(x+y)$ then $u_{y}\in H^{1}(\R^{N})$ with $u_{y}^{+}(x)=u^{+}(x+y)$ and $u_{y}^{-}(x)=u^{-}(x+y)$.
\end{lemma}

\begin{proof}
	Define
	$$\begin{array}{rcccl}
	T: & H^{1}(\R^{N}) & \rightarrow & H^{1}(\R^{N}) \\
	&     u         & \mapsto     & u_{y}
	\end{array}$$
	such that $u_{y}(x):=u(x+y)$ for all $x\in\R^{N}$. A direct computation gives $T(E^{+})\subset E^{+}$ and $T(E^{-})\subset E^{-}$. Consequently, 
	$$
	u(x+y)=u^{+}(x+y)+u^{-}(x+y)
	$$
or equivalently 
$$
T(u)=T(u^{+})+T(u^{-}).
$$
Since $T(u^{+})\in E^{+}$ and $T(u^{-})\in E^{-}$, we derive that $T(u)^{+}=T(u^{+})$ and $T(u)^{-}=T(u^{-})$, obtaining the desired result. 
\end{proof}

\vspace{0.5 cm}
The next lemma is a more weaker version of \cite[Lemma 2.5]{SW}.  

\begin{lemma}\label{adaptacao2.5}
Let  $\mathcal{V}\subset E^{+}\setminus\{0\}$ be a bounded set with $0\notin\overline{\mathcal{V}}^{\sigma(H^{1}(\R^{N}),H^{1}(\R^{N})')}$, that is, $0$ does not belong to the weak closure of $\mathcal{V}$. If  $W \in C(\mathbb{R}^N) \cap L^{\infty}(\mathbb{R}^N)$ with $\inf_{x \in \R^{N}}W(x)=W_0>0$ and $F:\R\rightarrow \R$ be a continuous function verifying
	\begin{itemize}
		\item[$(i)$] $\frac{F(t)}{t^{2}}\rightarrow +\infty$ as $|t|\rightarrow+\infty$.
		\item[$(ii)$] $F(t)\geq0$ for all $t\in\R$,
	\end{itemize}
 the functional $\varphi:H^{1}(\R^{N})\rightarrow\R\cup\{-\infty\}$ given by 
	$$
	\varphi(u)=\frac{1}{2}||u^{+}||^{2}-\frac{1}{2}||u^{-}||^{2}-\int W(x)F(u)dx,
	$$
satisfies $\varphi(u)<0$ on $\widehat{E}(u)\setminus B_{R}(0),\ $for all $u\in\mathcal{V}$ and for some $R>0$. 
\end{lemma}
\begin{proof}
	Suppose by contradiction that there exist $(u_{n}) \subset \mathcal{V}$ and $(w_{n}) \subset \widehat{E}(u_{n})\setminus B_{n}(0)$ with $\varphi(w_{n})\geq0$. 
	As $||w_{n}||\rightarrow+\infty$, we set $v_{n}:=\frac{w_{n}}{||w_{n}||}\in\widehat{E}(u_{n})$. Then, there is $s_{n}\geq0$  such that
	$$
	v_{n}=s_{n}u_{n}+v_{n}^{-}.
	$$
	Consequently $w_{n}=||w_{n}||s_{n}u_{n}+||w_{n}||v_{n}^{-}$ and 
	\begin{equation}\label{desigualdadecontraditoria}
	0\leq\frac{\varphi(w_{n})}{||w_{n}||^{2}}=\frac{1}{2}s_{n}^{2}||u_{n}||^{2}-\frac{1}{2}||v_{n}^{-}||^{2}-\int\frac{W(x)F(w_{n})}{||w_{n}||^{2}}dx.
	\end{equation}
	From this, $s_{n}u_{n}\not\rightarrow0$. In fact, otherwise, $s_{n}||u_{n}||\rightarrow0$ leads to 
	$$
	0 \leq \frac{1}{2}||v_{n}^{-}||^{2}+\int\frac{W(x)F(w_{n})}{||w_{n}||^{2}}dx\leq\frac{1}{2}s_{n}^{2}||u_{n}||^{2}\rightarrow 0.
	$$
	Therefore $v_{n}^{-}\rightarrow0$ and $v_{n}=s_{n}u_{n}+v_{n}^{-}\rightarrow0$, which is absurd, because $||v_{n}||=1$ for all $n\in\N$. Thereby, $s_{n}u_{n}\not\rightarrow0$. As   $(u_{n})$ is bounded, we have $s_{n}\not\rightarrow0$. On the other hand, since $0\notin\overline{\mathcal{V}}^{\sigma(H^{1}(\R^{N}),H^{1}(\R^{N})')}$, it follows that $u_{n}\not\rightharpoonup0$, and so, $u_{n}\not\rightarrow0$. Since $s_{n}^{2}||u_{n}||^{2}\leq||v_{n}||^{2}=1$, we conclude that $s_{n}\not\rightarrow+\infty$. Thus, for some subsequence, 
	$s_{n}\rightarrow s\neq0$, $u_{n}\rightharpoonup u\neq0$ and 
	$$
	v_{n}=s_{n}u_{n}+v_{n}^{-}\rightharpoonup v=su+v^{-}\neq0.
	$$
 So, by Fatou's Lemma, 
	$$
	\int\frac{W(x)F(w_{n})}{||w_{n}||^{2}}dx\geq\int\frac{W(x)F(w_{n})}{|w_{n}|^{2}}|v_{n}|^{2}dx\geq\int_{[v\neq0]}\frac{W(x)F(w_{n})}{|{w_{n}|^{2}}}|v_{n}|^{2}dx\rightarrow+\infty,
	$$
contradicting  (\ref{desigualdadecontraditoria}).
\end{proof}

After the above commentaries we are ready to prove the main result this section.

\begin{proposition}\label{continuidadec}  The function $\lambda\mapsto d_{\lambda}$ is decreasing and continuous on $(0,+\infty)$.
\end{proposition}
\begin{proof} In the sequel, $u_\lambda$ and $u_\mu$ are ground state solutions of $J_\lambda$ and $J_\mu$ respectively. Note that if $\lambda>\mu$, then 
$$
J_{\mu}(u)-J_{\lambda}(u)=(\lambda-\mu)\int F(u)\,dx\geq 0, \quad \forall u\in H^{1}(\R^{N}).
$$
Hence
$$
d_{\lambda}=\inf_{u\in E^{+}\setminus\{0\}}\max_{v\in\widehat{E}(u)}J_{\lambda}(u)\leq\inf_{u\in E^{+}\setminus\{0\}}\max_{v\in\widehat{E}(u)}J_{\mu}(u)=d_{\mu},
$$
showing that $\lambda\mapsto d_{\lambda}$ is monotone non-increasing. We claim that $d_{\lambda}<d_{\mu}$. Indeed, suppose $d_\lambda = d_\mu$ and let $t_{\mu}\geq0$ and $v_{\mu}\in E^{-}$ satisfying  
$$
J_{\lambda}(t_{\mu}u_{\mu}+v_{\mu})=\displaystyle\max_{u \in \widehat{E}(u_{\mu})}J_{\lambda}(u). \quad ( \mbox{see} \,\, (\ref{maximo}) )
$$ 
Therefore,
$$
\begin{array}{ll}
d_{\lambda}& \leq J_{\lambda}(t_{\mu}u_{\mu}+v_{\mu}) =(\mu-\lambda)\int F(t_{\mu}u_{\mu}+v_{\mu})\,dx+ J_{\mu}(t_{\mu}u_{\mu}+v_{\mu}) \\
& \leq(\mu-\lambda)\int F(t_{\mu}u_{\mu}+v_{\mu})\,dx+J_{\mu}(u_{\mu})\\
&=(\mu-\lambda)\int F(t_{\mu}u_{\mu}+v_{\mu})\,dx+d_{\mu}.
\end{array}
$$
As $d_\lambda=d_\mu$, it follows that
$$
(\mu-\lambda)\int F(t_{\mu}u_{\mu}+v_{\mu})\,dx \geq 0.
$$
By using the fact that $\lambda > \mu$ and $(f_4)$, we get $t_{\mu}u_{\mu}+v_{\mu}=0$ a.e. in $\R^{N}$, and so,  $d_{\lambda}\leq J_{\lambda}(t_{\mu}u_{\mu}+v_{\mu})=0$,  contradicting  (\ref{positivo}) . From this, the function $\lambda\mapsto d_{\lambda}$ is injective and decreasing. 

Now we are going to prove the continuity of $\lambda\mapsto d_{\lambda}$. To this end, we will divide into two steps our arguments: \\
\noindent\textbf{Step 1: }Let $(\lambda_{n})$ be a sequence with $\lambda_{1}\leq\lambda_{2}\leq...\leq\lambda_{n} \to \lambda$. Our goal is to prove that $\lim_{n \to +\infty}d_{\lambda_{n}}=d_{\lambda}$. Since $\lambda\mapsto d_{\lambda}$ is decreasing then $d_{\lambda}\leq d_{\lambda_{n}},\ \forall\ n\in\N$. For each $n \in \mathbb{N}$, let us fix $t_{n} \geq 0$  and $v_{n} \in E^{-}$ verifying 
$$
J_{\lambda_{n}}(t_{n}u_{\lambda}+v_{n})=\max_{u \in \widehat{E}(u_{\lambda})}J_{\lambda_{n}}(u).
$$
From Lemma \ref{adaptacao2.5}, there exists $R>0$ such that $J_{\lambda_{1}}(u)\leq0$ for all $u\in\widehat{E}(u_{\lambda})\setminus B_{R}(0)$. Recalling that  $J_{\lambda_{n}}\leq J_{\lambda_{1}}$, we have 
\begin{equation}\label{permanecenabola}
J_{\lambda_{n}}(u)\leq0,\ \forall\ u\in\widehat{E}(u_{\lambda})\setminus B_{R}(0) \quad \mbox{and} \quad \forall\ n\in \N.
\end{equation}
On the other hand $J_{\lambda_{n}}(t_{n}u_{\lambda}+v_{n})=\displaystyle\max_{u \in \widehat{E}(u_{\lambda})}J_{\lambda_{n}}(u)\geq d_{\lambda_{n}}\geq d_{\lambda}>0$, i. e.,
\begin{equation}\label{permanecenabola2}
J_{\lambda_{n}}(t_{n}u_{\lambda}+v_{n})>0,\ \ \forall\ n\in\N.
\end{equation}
By (\ref{permanecenabola}) and (\ref{permanecenabola2}), $||t_{n}u_{\lambda}+v_{n}||\leq R$ for all $n\in\N$. Then, $(t_{n}u_{\lambda}+v_{n})$ is bounded in $H^{1}(\R^{N})$ and 
$$
\begin{array}{ll}
d_{\lambda_{n}}& \leq J_{\lambda_{n}}(t_{n}u_{\lambda}+v_{n})\\
& =(\lambda-\lambda_{n})\int F(t_{n}u_{\lambda}+v_{n})dx +J_{\lambda}(t_{n}u_{\lambda}+v_{n}) \\
& \leq(\lambda-\lambda_{n})\int F(t_{n}u_{\lambda}+v_{n})dx +J_{\lambda}(u_{\lambda})=o_{n}(1)+d_{\lambda}.
\end{array}
$$
From this, 
$$
d_{\lambda_{n}}\leq o_{n}(1)+d_{\lambda} \quad \mbox{and} \quad d_{\lambda} \leq d_{\lambda_{n}}, \quad \forall n \in \mathbb{N},
$$ 
implying that $\displaystyle\lim_{n \to +\infty}d_{\lambda_{n}}=d_{\lambda}$.

\noindent\textbf{Step 2:} Let $(\lambda_{n})$ be a sequence with  $\lambda_{1}\geq\lambda_{2}\geq...\geq\lambda_{n} \to \lambda$. Our goal is to prove $\lim_{n \to +\infty}d_{\lambda_{n}}=d_{\lambda}$. Since $\lambda\mapsto d_{\lambda}$ is decreasing then $d_{\lambda_{1}}\leq d_{\lambda_{n}}\leq d_{\lambda}$, for all $n\in\N$. From \cite{SW}, for each $n\in\N$ let $u_{n}$ be a ground state solution of $(AP)_{\lambda_{n}}$, $t_{n}\geq0$ and $v_{n}\in E^{-}$ verifying
$$
J_{\lambda}(t_{n}u_{n}+v_{n})=\max_{u \in\widehat{E}(u_{n})}J_{\lambda}(u).
$$

Our next goal is to show that $(u_{n})$ is bounded. Inspired by \cite[Proposition 2.7]{SW}, suppose by contradiction that $||u_{n}||\to +\infty$ and let $w_{n}:=\frac{u_{n}}{||u_{n}||}$. As $||u_{n}^{+}||\geq||u_{n}^{-}||$, then $||w_{n}^{+}||^{2}\geq||w_{n}^{-}||^{2}$. Using the equality $||w_{n}^{+}||^{2}+||w_{n}^{-}||^{2}=||w_{n}||^{2}=1$, we derive $||w_{n}^{+}||^{2}\geq {1}/{2},\ \forall\ n\in\N$. Consequently there exist $(y_{n}) \subset \Z^{N}$ and $r,\eta>0$ such that
\begin{equation}\label{desigualdadelions}
\int_{B_{r}(y_{n})}|w_{n}^{+}(x)|^{2}dx\geq\eta,\ \ \forall\ n\in\N.
\end{equation}
Otherwise, we can apply Lions \cite[Lemma I.1]{lionsIII} to conclude that $w_{n}^{+}\rightarrow0$ in $L^{p}(\R^{N})$ for $p \in (2,2^*)$. Then,  $\int F(sw_{n}^{+})dx \to 0$ for each $s>0$ and  
$$
\begin{array}{ll}
d_{\lambda}\geq & d_{\lambda_{n}}=J_{\lambda_{n}}(u_{n})\geq J_{\lambda_{n}}(sw_{n}^{+})=\frac{1}{2}s^{2}||w_{n}^{+}||^{2}-\lambda_{n}\int F(sw_{n}^{+})dx \\
& \geq\frac{s^{2}}{4}-\lambda_{n}\int F(sw_{n}^{+})dx \to \frac{s^{2}}{4},
\end{array}
$$
which is absurd because $s$ is arbitrary, which shows (\ref{desigualdadelions}). Now, we set 
$$
\widetilde{u}_{n}(x):=u_{n}(x+y_{n}) \quad \mbox{and} \quad \widetilde{w}_{n}(x):=w_{n}(x+y_{n}).
$$
By Lemma \ref{translacao}, $\widetilde{w}_{n}^{+}(x)=w_{n}^{+}(x+y_{n})$. Moreover, by (\ref{desigualdadelions}), $\widetilde{w}_{n}\rightharpoonup w$ with $w^{+}\neq0$,  because $\widetilde{w}_{n}^{+}\rightharpoonup w^{+}$. Since $\widetilde{u}_{n}=\widetilde{w}_{n}||u_{n}||$, it follows that $|\widetilde{u}_{n}(x)|\to +\infty$ for each $x\in\R^{N}$ with $w(x)\neq0$. Therefore, by Fatou's Lemma, 
$$
\int \frac{F(\widetilde{u}_{n})}{|\widetilde{u}_{n}|^{2}}|\widetilde{w}_{n}|^{2}dx\to+\infty.
$$
From this,
$$
\begin{array}{ll}
0 & \leq\frac{J_{\lambda_{n}}(u_{n})}{||u_{n}||^{2}}=\frac{1}{2}||w_{n}^{+}||^{2}-\frac{1}{2}||w_{n}^{-}||^{2}-\lambda_n\int\frac{F(u_{n})}{|u_{n}|^{2}}|w_{n}|^{2}dx \\
& =\frac{1}{2}||w_{n}^{+}||^{2}-\frac{1}{2}||w_{n}^{-}||^{2}-\lambda_n\int\frac{F(\widetilde{u}_{n})}{|\widetilde{u}_{n}|^{2}}|\widetilde{w}_{n}|^{2}dx \to -\infty
\end{array}
$$
obtaining a contradiction. This proves that $(u_{n})$ is bounded.

Now, we are ready to prove that $\displaystyle\lim_{n\to +\infty}d_{\lambda_{n}}=d_{\lambda}$. First of all, there exists $\eta>0$ such that
\begin{equation}\label{desigualdadelions2}
\max_{y\in\R^{N}}\int_{B_{1}(y)}|u_{n}^{+}(x)|^{2}dx\geq\eta,\ \ \forall\ n\in\N.
\end{equation}
Otherwise, Lions \cite[Lemma I.1]{lionsIII} ensures that $u_{n}^{+} \to 0$ in $L^{p}(\R^{N}),\ \forall\ p\in(2,2^{*})$, and so, $\int f(u_{n})u_{n}^{+}dx \to 0$. Now, combining this limit with the equality $0=J_{\lambda_{n}}'(u_{n})u_{n}^{+}$,   
we derive that $||u_{n}^{+}||\to 0$. However, by \cite{SW}, we know that  $||u_{n}^{+}||\geq\sqrt{2d_{\lambda_{n}}}\geq\sqrt{2d_{\lambda_{1}}}$ for all $n\in\N$, which is absurd. This proves   (\ref{desigualdadelions2}), and so, there exist $(y_{n}) \subset \Z^{N}$ and $r>0$ such that
$$
\int_{B_{r}(y_{n})}|u_{n}^{+}(x)|^{2}dx\geq\eta.
$$
Defining $\widetilde{u}_{n}(x):=u_{n}(x+y_{n})$, we have that $(\widetilde{u}_{n})$ is bounded and $\widetilde{u}_{n_j}^{+}\not\rightharpoonup0$ as $n_j\rightarrow+\infty$ for any subsequence. Fixing $\mathcal{V}:=\{\widetilde{u}_{n}^{+}\}_{n\in\N}\subset E^{+}\setminus\{0\}$, it follows that $\mathcal{V}$ is bounded and $0\notin\overline{\mathcal{V}}^{\sigma(H^{1}(\R^{N}),H^{1}(\R^{N})')}$. Thus, by Lemma \ref{adaptacao2.5}, there exists $R>0$ such that
\begin{equation}\label{adaptacaolema}
J_{\lambda}(w)<0\text{ for } w \in E(u)\setminus B_{R}(0),\ \forall\ u\in\mathcal{V}.
\end{equation}
On the other hand, if $\widetilde{v}_{n}(x):=v_{n}(x+y_{n})$, we have
\begin{equation}\label{positivonovalor}
J_{\lambda}(t_{n}\widetilde{u}_{n}+\widetilde{v}_{n})=J_{\lambda}(t_{n}u_{n}+v_{n})=\max_{u \in \widehat{E}(u_{n})}J_{\lambda}(u) \geq d_{\lambda}>0,\ \forall\ n\in\N.
\end{equation}
By (\ref{adaptacaolema}) and (\ref{positivonovalor}), it follows that $||t_{n}\widetilde{u}_{n}+\widetilde{v}_{n}||\leq R$, for all $n\in\N$. Therefore $||t_{n}u_{n}+v_{n}||\leq R$, for all $n\in\N$, that is, $(t_{n}u_{n}+v_{n})$ is bounded. Finally,
$$
\begin{array}{ll}
d_{\lambda} & \leq J_{\lambda}(t_{n}u_{n}+v_{n})\\
& =(\lambda_{n}-\lambda)\int F(t_{n}u_{n}+v_{n})dx
+J_{\lambda_{n}}(t_{n}u_{n}+v_{n}) \\
& \leq o_{n}(1)+J_{\lambda_{n}}(u_{n})=o_{n}+d_{\lambda_{n}},
\end{array}
$$
that is, 
$$
d_{\lambda}\leq o_{n}(1)+d_{\lambda_{n}},\ \forall\ n\in\N.
$$ 
Since $d_{\lambda}\geq d_{\lambda_{n}}$ for all $n\in\N$, we have $\displaystyle\lim_{n\to +\infty}d_{\lambda_{n}}=d_{\lambda}$, finishing the proof. 
\end{proof}

\section{Existence of ground state for problem $(P)_\epsilon$.}

In this section our main goal is proving that $c_\epsilon$ given in (\ref{cepsilon}) is a critical level for $I_\epsilon$ when $\epsilon$ is small enough. Hereafter, for each $\epsilon \geq 0 $, we denote by $I_\epsilon: H^{1}(\R^N) \to \mathbb{R}$ the energy functional associated with $(P)_\epsilon$ given by
$$
I_\epsilon(u)=\frac{1}{2}\int (|\nabla u|^{2}+V(x)|u|^{2}\,dx)-\int A(\epsilon x)F(u)\,dx, 
$$
or equivalently 
$$
I_\epsilon(u)=\frac{1}{2}\|u^+\|^{2}-\frac{1}{2}\|u^-\|^{2} -\int A(\epsilon x) F(u)\,dx.
$$

Here, it is very important to observe that by using the notations explored in Section 2, we have
$$
c_0=d_{A(0)}, \quad I_0=J_{A(0)} \quad \mbox{and} \quad \mathcal{M}_0=\mathcal{N}_{A(0)}. 
$$

The same idea explored in \cite[Lemma 2.4]{SW} gives
\begin{equation} \label{cepsilon2}
0<c_\epsilon=\inf_{u\in E^{+}\setminus\{0\}}\max_{v\in\widehat{E}(u)}I_{\epsilon}(u).
\end{equation}
Moreover, the Lemma \ref{adaptacao2.5} permits to argue as in \cite[Lemma 2.6]{SW} to prove that for each $u\in H^{1}(\R^{N})\setminus E^{-}$, $\mathcal{M}_\epsilon \cap\hat{E}(u)$ is a singleton set and the element of this set is the unique global maximum of $I_{\epsilon}|_{\hat{E}(u)}$, that is, there are $t_* \geq 0$ and $v_* \in E^{-}$ such that
\begin{equation} \label{maximo2}
I_{\epsilon}(t_*u+v_*)=\displaystyle\max_{w \in \widehat{E}(u)}I_{\epsilon}(w).
\end{equation}

Our first lemma shows an important relation between $c_\epsilon$ and $c_0$.  
\begin{lemma}\label{limitce}
It occurs the limit $\displaystyle\lim_{\epsilon\rightarrow0}c_{\epsilon}=c_{0}$.
\end{lemma}
\begin{proof}
 Consider $\epsilon_{n}\rightarrow0$ with $\epsilon_{n}>0$. Our goal is to prove that $c_{\epsilon_{n}}\rightarrow c_{0}$. First of all, note that $c_{0}\leq c_{\epsilon_{n}}$ for all $n\in\N$, which leads to  $c_{0}\leq\liminf_{n \to +\infty}c_{\epsilon_{n}}$. On the other hand, by (\ref{maximo2}), if $w_{0}\in H^{1}(\R^{N})$ is a ground state solution of  $(P)_{0}$, there are $t_{n}\in[0,+\infty)$ and $v_{n}\in E^{-}$ such that $t_{n}w_{0}^{+}+v_{n}\in\M_{\epsilon_{n}}$, and so,  
 $$
 I_{\epsilon_{n}}(t_{n}w_{0}^{+}+v_{n})\geq c_{\epsilon_{n}}>0,\ \ \forall\ n \in \mathbb{N}.
 $$
As in the previous section, $(t_{n}w_{0}^{+}+v_{n})$ is bounded. Thus, without loss of generality, we can assume  that $t_{n}\rightarrow t_{0}$ and $v_{n}\rightharpoonup v$ in $H^{1}(\R^N)$. Note that
$$
c_{\epsilon_{n}}\leq I_{\epsilon_{n}}(t_{n}w_{0}^{+}+v_{n})=\frac{1}{2}t_{n}^{2}||w_{0}^{+}||^{2}-\frac{1}{2}||v_{n}||^{2}-\int A(\epsilon_{n} x)F(t_{n}w_{0}^{+}+v_{n})dx.
$$
Hence, since the norm is weakly lower semicontinous, the Fatou's Lemma gives   
$$
\begin{array}{ll}
\limsup_{n \to +\infty}c_{\epsilon_{n}} & \leq \limsup_{n\to +\infty}\left(\frac{1}{2}t_{n}^{2}||w_{0}^{+}||^{2}-\frac{1}{2}||v_{n}||^{2}\right)+ \\ 
& \limsup_{n \to +\infty}\left(-\int A(\epsilon_{n} x)F(t_{n}w_{0}^{+}+v_{n})dx\right) \\
& \leq  \frac{1}{2}t_{0}^{2}||w_{0}^{+}||^{2}-\frac{1}{2}||v||^{2}-\int A(0)F(t_{0}w_{0}^{+}+v)dx \\
& =I_{0}(t_{0}w_{0}^{+}+v)\leq I_{0}(w_{0})=c_{0}.
\end{array}
$$
Therefore, $\lim_{n \to +\infty}c_{\epsilon_{n}}=c_{0}$.
\end{proof}

As an immediate consequence of the last lemma we have the corollary below

\begin{corollary}\label{epsilonpequeno}
	There exists $\epsilon_{0}>0$ such that $c_{\epsilon}<d_{A_\infty}$ for all $\epsilon\in(0,\epsilon_{0})$, where $A_\infty=\lim_{|x| \to +\infty}A(x)$.
\end{corollary}

\begin{proof}
By condition $(A_1)$, $A(0) > A_\infty$, then the Proposition \ref{continuidadec} ensures that $d_{A(0)}<d_{A_\infty}$, or equivalently, $c_{0}<d_{A_\infty}$. Now it is enough to apply  the Lemma \ref{limitce} to get the desired result. 
\end{proof}

\vspace{0.5 cm}

As a byproduct of the proof of  Lemma \ref{limitce}, we also have the following result, which can be useful for related problems.
\begin{lemma} Let $(t_n) \subset [0,+\infty)$ and $(v_n) \subset E^{-}$ be the sequences defined in the proof of Lemma \ref{limitce}. Then, for some subsequence, 
$$
t_n \to 1 \quad \mbox{and} \quad v_n \to w_0^-.   
$$
Hence, $t_n w_0^++v_n \to w_0$ in $H^{1}(\R^N)$. 
\end{lemma}
\begin{proof}
Note that in the proof of Lemma \ref{limitce},  we find that  
$$
\liminf_{n\to +\infty} ||v_{n}||^{2}=||v||^{2}.
$$
Then for some subsequence $\lim_{n \to +\infty}||v_{n}||=||v||$, and so, $v_{n}\rightarrow v$. Furthermore, from the previous lemma  $I_{0}(w_{0})=I_{0}(t_{0}w_{0}^{+}+v)$, where $w_{0}\in\M_{0}$. Hence $t_{0}w_{0}^{+}+v=w_{0}$, from where it follows that $t_{0}=1$ and $v=w_{0}^{-}$. Thereby, $t_{n}\rightarrow 1$ and $v_{n}\rightarrow w_{0}^{-}$.
\end{proof}

\vspace{0.5 cm}

Our next result is related to the \cite[Proposition 2.7]{SW}, however as in the present paper $A$ is not periodic, we cannot repeat the same arguments explored in that paper, then some adjustments are necessary in the proof to get the same result.  
\begin{proposition}\label{limitacao}
$I_{\epsilon}$ is coercive on $\M_{\epsilon}$.
\end{proposition}
\begin{proof}
Suppose that there exists $(u_{n}) \subset \M_{\epsilon}$ verifying 
$$
I_{\epsilon}(u_{n})\leq d \quad \mbox{and} \quad ||u_{n}||\rightarrow+\infty,
$$ 
for some $d \in \R$. Setting $v_{n}:=\frac{u_{n}}{||u_{n}||}$, it follows that $||v_{n}^{+}||\geq||v_{n}^{-}||$ and $||v_{n}^{+}||^{2}\geq\frac{1}{2}$. On the other hand, there exist $(y_{n}) \subset \Z^{N}$ and $r,\eta>0$ such that,
\begin{equation}\label{vnlions}
\int_{B_{r}(y_{n})}|v_{n}^{+}|^{2}dx>\eta, \quad \forall n\in\N.
\end{equation}
In fact, suppose by contradiction that (\ref{vnlions}) does not hold. Then, applying again Lions \cite[Lemma I.1]{lionsIII}, $v_{n}^{+}\rightarrow0$ in $L^{p}(\R^{N})$ for all $p\in(2,2^{*})$. Hence,  by $(f_1)-(f_2)$, $\int F(sv_{n}^{+})dx\rightarrow0$ for all $s>0$. Thereby, 
$$
\begin{array}{ll}
d & \geq  I_{\epsilon}(u_{n})\geq I_{\epsilon}(sv_{n}^{+})=\frac{1}{2}s^{2}||v_{n}^{+}||^{2}-\int A(\epsilon x)F(sv_{n}^{+})dx\geq \\
& \geq\frac{s^{2}}{4}-\int A(0)F(sv_{n}^{+})dx\rightarrow\frac{s^{2}}{4},
\end{array}
$$
which absurd, because $s$ is arbitrary. This shows that (\ref{vnlions}) is valid.

Fixing  $\widetilde{u}_{n}(x):=u_{n}(x+y_{n})$ and $\widetilde{v}_{n}(x):=v_{n}(x+y_{n})$,  by Lemma \ref{translacao}, we have $\widetilde{v}^{+}_{n}(x):=v^{+}_{n}(x+y_{n})$ and $\widetilde{u}_{n}=\widetilde{v}_{n}||u_{n}||$. Since $\widetilde{v}_{n}\rightharpoonup v$, by (\ref{vnlions}), $v\neq0$. Then, $\widetilde{u}_{n}(x)\rightarrow+\infty$ when $v(x)\neq0$. By using  the Fatou's Lemma, we get
$$
\int\frac{F(u_{n})}{||u_{n}||^{2}}\,dx\geq\int\frac{F(u_{n})}{|u_{n}|^{2}}|v_{n}|^{2}dx=\int\frac{F(\widetilde{u}_{n})}{|\widetilde{u}_{n}|^{2}}|\widetilde{v}_{n}|^{2}dx\geq\int_{[v\neq0]}\frac{F(\widetilde{u}_{n})}{|\widetilde{u}_{n}|^{2}}|\widetilde{v}_{n}|^{2}dx\rightarrow+\infty.
$$
The above limit yields 
$$
\begin{array}{ll}
0 & \leq\frac{I_{\epsilon}(u_{n})}{||u_{n}||^{2}}=\frac{1}{2}||v_{n}^{+}||^{2}-\frac{1}{2}||v_{n}^{-}||^{2}-\int A(\epsilon x)\frac{F(u_{n})}{||u_{n}||^{2}}dx \\ 
& \leq \frac{1}{2}-A_0\int\frac{F(u_{n})}{||u_{n}||^{2}}dx\rightarrow-\infty,
\end{array}
$$
obtaining a new absurd.
\end{proof}

\vspace{0.5 cm}

Now, we can repeat the same arguments found in \cite[see proof of Theorem 1.1]{SW} to guarantee the existence of a $(PS)$ sequence $(u_{n}) \subset \mathcal{M}_\epsilon$ associated with $c_{\epsilon}$, that is, 
$$
I_\epsilon(u_n) \to c_\epsilon \quad \mbox{and} \quad I'_\epsilon(u_n) \to 0.
$$

\begin{theorem}
The problem $(P)_{\epsilon}$ has a ground state solution for all $\epsilon \in (0, \epsilon_0)$, where $\epsilon_0 >0$ was given in Corollary \ref{epsilonpequeno}.
\end{theorem}
\begin{proof}
First of all, the fact that $(u_{n}) \subset \mathcal{M}_\epsilon$ leads to 
$$
0=I_{\epsilon}'(u_{n})u_{n}^{+}=||u_{n}^{+}||^{2}-\int A(\epsilon x)f(u_{n})u_{n}^{+}dx\geq 2c_{\epsilon}-\int A(\epsilon x)f(u_{n})u_{n}^{+}dx.
$$
Therefore $\int A(\epsilon x)f(u_{n})u_{n}^{+}dx\not\rightarrow0$. Since $(u_{n})$ is bounded, by Lions  \cite[Lemma I.1]{lionsIII}, there exist $\eta,\delta>0$ and $(z_{n}) \subset \Z^{N}$ such that
$$
\int_{B_{\delta}(z_{n})}|u_{n}^{+}|^{2}dx>\eta, \quad \forall n \in \mathbb{N}.
$$
\begin{claim} 
	$(z_n)$ is a bounded sequence.
\end{claim}
If $(z_{n})$ is unbounded, for some subsequence, we must have $|z_{n}|\rightarrow+\infty$. Fixing  $w_{n}(x):=u_{n}(x+z_{n})$, we derive $w_{n}\rightharpoonup w\neq0$. Now, for each  $\phi\in C_{0}^{\infty}(\R^{N})$,  
$$
\begin{array}{ll}
o_{n}(1)& =I_{\epsilon}'(u_{n})\phi(\cdot-z_{n})=B(u_{n},\phi(\cdot-z_{n}))-\int A(\epsilon x)f(u_{n})\phi(\cdot-z_{n})dx \\
& =B(w_{n},\phi)-\int A(\epsilon x+\epsilon z_{n})f(w_{n})\phi dx,
\end{array}
$$
where
$$
B(u,v)=\int (\nabla u \nabla v+V(x)uv)\,dx, \quad \forall u,v \in H^{1}(\R^N).
$$
Taking the limit  $n\rightarrow+\infty$,  we obtain
$$
0=B(w,\phi)-\int A_{\infty}f(w)\phi dx=J_{A_\infty}'(w)\phi,\ \forall\ \phi\in C^{\infty}_{0}(\R^{N}).
$$
Now, the density of $C^{\infty}_{0}(\R^{N})$ in $H^{1}(\R^N)$ gives
$$
0=B(w,v)-\int A_{\infty}f(w)v dx=J_{A_\infty}'(w)v,\ \forall\ v\in H^{1}(\R^N).
$$
The last equality says that $w$ is a nontrivial solution of $(AP)_{A_\infty}$. The characterization of $d_{A_\infty}$ together with  Fatou's Lemma yields  
$$
\begin{array}{ll}
d_{A_\infty}& \leq J_{A_\infty}(w)=J_{A_\infty}(w)-\frac{1}{2}J_{A_\infty}'(w)w=\int A_{\infty}\left(\frac{1}{2}f(w)w-F(w)\right)dx \\
& \leq \liminf_{n\to +\infty}\int A(\epsilon x+\epsilon z_{n})\left(\frac{1}{2}f(w_{n})w_{n}-F(w_{n})\right)dx \\
& =\liminf_{n \to +\infty}\int A(\epsilon x)\left(\frac{1}{2}f(u_{n})u_{n}-F(u_{n})\right)dx \\ 
& =\liminf_{n\to +\infty}\left(I_{\epsilon}(u_{n})-\frac{1}{2}I_{\epsilon}'(u_{n})u_{n}\right)=c_{\epsilon},
\end{array}
$$
that is 
$$
d_{A_\infty}\leq c_{\epsilon}, \quad \forall \epsilon >0.  
$$ 
On the other hand, by Corollary \ref{epsilonpequeno}, $c_{\epsilon}<d_{A_\infty}$ when $\epsilon<\epsilon_{0}$,  which is absurd. Therefore, $(z_{n})$ is bounded.

As  $(z_{n})$ is bounded, there exists $r>0$ such that $B_{\delta}(z_{n})\subset B_{r}(0)$ for all $n\in\N$. Then, 
$$
\int_{B_{r}(0)}|u_{n}^{+}|^{2}dx\geq\int_{B_{\delta}(z_{n})}|u_{n}^{+}|^{2}dx>\eta, \quad \forall n \in \mathbb{N}.
$$
From this, $u_{n}\rightharpoonup u$ with $u\neq0$. Now, it is enough to repeat the arguments found \cite[page 23]{AG} to conclude that $u$ is a ground state solution for $(P)_{\epsilon}$. \end{proof}

\section{Concentration of maxima}

In this section, we denote by $u_\epsilon$ the ground state solution obtained in Section 3. Our main goal is to show that if $x_\epsilon$ is a maximum point of $|u_\epsilon|$, then 
$$
\lim_{\epsilon \to 0}A(\epsilon x_\epsilon)=A(0).
$$  
Of a more precise way, we must prove that if $\epsilon_n \to 0$,  for some subsequence, then $\epsilon_n x_{\epsilon_n} \to x_0$ for some $x_0 \in \mathcal{A}$ where
$$
\mathcal{A}=\{z \in \R^N\,:\,A(z)=A(0)\}.
$$

In what follows, we set  $(\epsilon_{n}) \subset (0,\epsilon_0)$ with $\epsilon_{n}\rightarrow0$, $I_n=I_{\epsilon_n}$, $c_{n}:=c_{\epsilon_{n}}$ and $u_{n}=u_{\epsilon_n}$, that is,
$$
I'_{n}(u_n)=0 \quad \mbox{and} \quad I_{n}(u_{n})=c_{n}.
$$  
By $(A_1)$,  $c_{n}\geq c_{0}>0$ for all $n\in\N$.

Next, we will show some technical lemmas that are crucial to get the concentration of maxima.  

\begin{lemma}
	The sequence $(u_{n})$ is bounded.
\end{lemma}
\begin{proof}
The proof follows as in Proposition \ref{limitacao}.
\end{proof}

\begin{lemma}
There exist $(y_{n}) \subset \Z^{N}$ and  $R$, $\eta>0$ verifying 
$$
\int_{B_{R}(y_{n})}|u_{n}^{+}|^{2}dx\geq\eta,\ \ \forall n\in\N.
$$ 
\end{lemma}
\begin{proof}
If the lemma does not hold, by Lions \cite[Lemma I.1]{lionsIII}, $u_{n}^{+}\rightarrow0$ in $L^{p}(\R^{N})$ for all $p\in(2,2^{*})$. Therefore 
$||u_{n}^{+}||^{2}=\int A(\epsilon_{n}x)f(u_{n})u_{n}^{+}dx\rightarrow0$. On the other hand, from \cite[Lemma 2.4]{SW}, we know that $||u_{n}^{+}||\geq\sqrt{2c_{n}}\geq\sqrt{2c_{0}}$, which contradicts the last limit. 
\end{proof}

\vspace{0.5 cm}

In the sequel, $v_{n}(x):=u_{n}(x+y_{n})$ for all $x\in\R^{N}$. Thus, for some subsequence,  $v_{n}\rightharpoonup v\neq0$. 

\begin{lemma}
The sequence $(\epsilon_{n}y_{n})$ is bounded in $\R^{N}$. Furthermore, if for a subsequence $\epsilon_{n}y_{n}\rightarrow z$, then $z \in \mathcal{A}, I'_0(v)=0$ and $I_0(v)=c_0$. Hence, $v_0$ is a ground state solution for $(P)_0$.
\end{lemma}

\begin{proof}
First of all, we will prove the boundedness of the sequence $(\epsilon_{n}y_{n})$. Arguing by contradiction, suppose that for some subsequence $|\epsilon_{n}y_{n}|\rightarrow+\infty$. Since $u_{n}$ is a ground state solution for $(P)_{\epsilon_{n}}$,
$$
\int (\nabla u_{n}\nabla\phi(x-y_{n})+V(x)u_{n}\phi(x-y_{n})) dx=\int A(\epsilon_{n}x)f(u_{n})\phi(x-y_{n})dx, 
$$
for all $\phi\in C^{\infty}_{0}(\R^{N})$. Hence, by a change variable, 
$$
\int (\nabla v_{n}\nabla\phi+V(x)v_{n}\phi) dx=\int A(\epsilon_{n}x+\epsilon_{n}y_{n})f(v_{n})\phi dx
$$
for all $\phi\in C_{0}(\R^{N})$. Now, taking the limit as $n\rightarrow+\infty$, we find
$$
\int (\nabla v\nabla\phi\,dx+V(x)v\phi) dx=\int A_{\infty}f(v)\phi dx
$$
for all $\phi\in C_{0}(\R^{N})$. This combined with the density of $C^{\infty}_{0}(\R^{N})$ in $H^{1}(\R^N)$ gives  
$$
\int(\nabla v\nabla\psi+V(x)v\psi) dx=\int A_{\infty}f(v) \psi dx, \quad \forall \psi \in H^{1}(\R^N). 
$$
Then $v$ is a nontrivial solution of $(AP)_{A_\infty}$, and so, $v\in\M_{A_\infty}$. By Fatou's lemma,
$$
\begin{array}{ll}
d_{A_\infty} & \leq J_{A_\infty}(v)=J_{A_\infty}(v)-\frac{1}{2}J_{A_\infty}'(v)v=\int A_{\infty}\left(\frac{1}{2}f(v)v-F(v)\right)dx \\  
& \leq\liminf_{n \to +\infty}\int A(\epsilon x+\epsilon_{n}y_{n})\left(\frac{1}{2}f(v_{n})v_{n}-F(v_{n})\right)dx \\
& =\liminf_{n \to +\infty}\int A(\epsilon_{n}x)\left(\frac{1}{2}f(u_{n})u_{n}-F(u_{n})\right)dx \\
& =\liminf_{n \to +\infty}\left(I_{n}(u_{n})-\frac{1}{2}I_{n}'(u_{n})u_{n}\right) \\
& =\liminf_{n \to +\infty}I_{n}(u_{n})=\lim_{n\in\N}c_{n}=c_{0}<d_{A_\infty},
\end{array}
$$
obtaining a contradiction. Consequently $(\epsilon_{n}y_{n})$ is bounded, and we can assume that $\epsilon_{n}y_{n}\rightarrow z$. The same argument works to prove that
$$
\int (\nabla v\nabla\psi+V(x)v\psi) dx=\int A(z)f(v)\psi dx, \quad \forall  \psi \in H^{1}(\R^N).
$$
Hence $v$ is a nontrivial solution of $(AP)_{A(z)}$, and so, $v\in\M_{A(z)}$. The previous arguments lead to $d_{A(z)}\leq c_{0}=d_{A(0)}$. Then the monotonicity of $\lambda\rightarrow d_{\lambda}$ implies that $A(0)\leq A(z)$. As $A(0) \geq A(z)$, it follows that $A(0)=A(z)$, that is, $ z \in \mathcal{A}$. Moreover, the analysis above also gives $I'_0(v)=0$ and $I_0(v)=c_0$. 
\end{proof}

\vspace{0.5 cm}

From now on, we are considering that $\epsilon_{n}y_{n}\rightarrow z$ with $z \in \mathcal{A}$, i.e., $A(z)=A(0)$. Here, it is very important to observe that  
$$
J_{A(z)}=J_{A(0)}=I_0, \quad I'_{0}(v)=0 \quad \mbox{and} \quad I_0(v)=c_0.
$$

By the growth conditions on $f$, we  know that for each $\tau >0$ there exists $\delta:=\delta_{\tau}\in(0,1)$ such that 
$$
\frac{|f(t)|}{|t|}<\tau, \quad \forall t\in(-\delta,\delta).
$$ 
In what follows, we fix $g_{\tau}(t):=\chi_{\delta}(t)f(t)$ and $j_{\tau}(t):=\tilde{\chi}_\delta(t)f(t)$, where $\chi_\delta$ is the characteristic function on $(-\delta,\delta)$ and $\tilde{\chi}_\delta(t)=1-\chi_\delta(t)$.

\begin{lemma}\label{geje}
For each $\tau>0$, there is $c_{\tau}>0$ such that 
$$
|g_{\tau}(t)|\leq\tau|t| \quad  \mbox{and} \quad |j_{\tau}(t)|^{r}\leq c_{\tau}tf(t), \quad  \forall t \in \mathbb{R},
$$ 
where $r=\frac{q+1}{q}$ with $q$ given in $(f_2)$.
\end{lemma}
\begin{proof}
By using the definition of $g_\tau$, it is obvious that  first inequality involving the function $g_\tau$ is true. 

In order to prove the second inequality, note that $[-1,-\delta]\cup[\delta,1]\subset\R$ is compact set, then there exists $\widetilde{c_{\tau}}>0$ such that 
$$
\frac{|f(t)|^{r-1}}{|t|}\leq\widetilde{c_{\tau}}, \quad \forall t\in[-1,-\delta]\cup[\delta,1],
$$ 
consequently 
$$
|j_{\tau}(t)|^{r-1}\leq\widetilde{c_{\tau}}|t|, \quad \forall t\in[-1,-\delta]\cup[\delta,1].
$$ 
On the other hand, there exists $\widetilde{b_{\tau}}>0$ verifying 
$$
|f(t)|\leq\tau|t|+\widetilde{b_{\tau}}|t|^{q},\ \forall\ t\in\R.
$$
Thus, there exist $A_{\tau}, B_{\tau}, \widehat{c_{\tau}}>0$ such that
$$
|j_{\tau}(t)|^{r-1}=|f(t)|^{r-1}\leq A_{\tau}|t|^{r-1}+B_{\tau}|t|^{(r-1)q}=A_{\tau}|t|^{r-1}+B_{\tau}|t|\leq\widehat{c_{\tau}}|t|, \quad \forall |t|>1. 
$$
From this,  
$$
|j_{\tau}(t)|^{r-1}\leq c_{\tau}|t|, \quad \forall t\in\R,
$$
for some $c_{\tau}>0$. Thereby, 
$$
|j_{\tau}(t)|^{r}\leq c_{\tau}|t||j_\tau(t)|\leq c_{\tau}tf(t), \quad \forall t\in\R,
$$
finishing the proof.
\end{proof}

The last lemma permits us to prove an important convergence involving the sequence $(v_n)$.

\begin{proposition}\label{convergenciaforte}
The sequence $(v_{n})$ converges strongly to $v$ in $H^{1}(\R^N)$.
\end{proposition}
\begin{proof}
First of all, note that
$$
\begin{array}{ll}
c_{0} \leq & I_{0}(v)=I_{0}(v)-\frac{1}{2}I'_{0}(v)v=\int A(0)\left(\frac{1}{2}f(v)v-F(v)\right)dx \\
& =\int A(z)\left(\frac{1}{2}f(v)v-F(v)\right)dx \\
& \leq\liminf_{n \to +\infty}\int A(\epsilon_{n} x+\epsilon_{n}y_{n})\left(\frac{1}{2}f(v_{n})v_{n}-F(v_{n})\right)dx \\ 
& \leq\limsup_{n \to +\infty}\int A(\epsilon_{n} x+\epsilon_{n}y_{n})\left(\frac{1}{2}f(v_{n})v_{n}-F(v_{n})\right)dx \\
& =\limsup_{n \to +\infty}\int A(\epsilon_{n}x)\left(\frac{1}{2}f(u_{n})u_{n}-F(u_{n})\right)dx \\
& =\limsup_{n \to +\infty}\left(I_{n}(u_{n})-\frac{1}{2}I'_{n}(u_{n})u_{n}\right) \\
&=\lim_{n \to +\infty}c_{n}=c_{0}.
\end{array}
$$
Therefore 
$$\lim_{n\to +\infty}\int A(\epsilon_{n} x+\epsilon_{n}y_{n})\left(\frac{1}{2}f(v_{n})v_{n}-F(v_{n})\right)dx=\int A(z)\left(\frac{1}{2}f(v)v-F(v)\right)dx.$$
Since 
$$
A(\epsilon_{n} x+\epsilon_{n}y_{n})\left(\frac{1}{2}f(v_{n})v_{n}-F(v_{n})\right)\geq 0, \quad \forall n \in \mathbb{N},
$$ 
and supposing that
$$
v_n(x) \to v(x) \quad \mbox{a.e. in} \quad \mathbb{R}^N,
$$
we deduce that
$$
A(\epsilon_{n} x+\epsilon_{n}y_{n})\left(\frac{1}{2}f(v_{n})v_{n}-F(v_{n})\right)\rightarrow A(z)\left(\frac{1}{2}f(v)v-F(v)\right) \quad \mbox{a.e. in} \quad \mathbb{R}^N.
$$
Fixing
$$
h_n(x)=A(\epsilon_{n} x+\epsilon_{n}y_{n})\left(\frac{1}{2}f(v_{n})v_{n}-F(v_{n})\right)\geq 0 \quad \mbox{and} \quad h(x)= A(z)\left(\frac{1}{2}f(v)v-F(v)\right)\geq 0
$$
we infer that 
$$
h_n(x) \to h(x)  \quad \mbox{a.e. in} \quad \mathbb{R}^N 
$$
and 
$$
|h_n|_1=\int_{\mathbb{R}^N}|h_n(x)|\,dx=\int_{\mathbb{R}^N}h_n(x)\,dx \to |h|_1=\int_{\mathbb{R}^N}|h(x)|\,dx=\int_{\mathbb{R}^N}h(x)\,dx. 
$$
However, by a result due to Br\'ezis-Lieb \cite{BL}, we know that 
$$
|h|_1=\lim_{n \to +\infty}(|h_n|_1-|h_n-h|_1), 
$$
from where it follows 
$$
\lim_{n \to +\infty}|h_n-h|_1=0, 
$$
that is, 
$$
h_n \to h \quad \mbox{in} \quad L^{1}(\mathbb{R}^N),
$$
or equivalently
$$
A(\epsilon_{n} x+\epsilon_{n}y_{n})\left(\frac{1}{2}f(v_{n})v_{n}-F(v_{n})\right) \to A(z)\left(\frac{1}{2}f(v)v-F(v)\right) \quad \mbox{in} \quad L^{1}(\mathbb{R}^N).
$$
Thus, for some subsequence, there exists $H \in L^{1}(\R^{N})$ such that 
$$
A_0 \left(\frac{1}{2}f(v_{n})v_{n}-F(v_{n})\right)\leq A(\epsilon_{n} x+\epsilon_{n}y_{n})\left(\frac{1}{2}f(v_{n})v_{n}-F(v_{n})\right)\leq H \quad \mbox{a.e. in} \quad \R^N
$$
for all $n\in\N$. Then, by $(f_4)$,
$$
A_0 \left(\frac{1}{2}-\frac{1}{\theta}\right)f(v_{n})v_{n}\leq H, \quad \forall n \in \mathbb{N}.
$$
Consequently there exists $c>0$ such that
$$
f(v_{n})v_{n}\leq c\,H, \quad \forall n \in \mathbb{N}. 
$$
In what follows, we set
$$
Q_{n}:=f(v_{n})v_{n}^{+}-f(v)v^{+}. 
$$
Our goal is to prove that 
$$
\int |Q_{n}|dx\rightarrow0.
$$ 
First of all, as $f$ has a subcritical growth,  
\begin{equation} \label{g}
\int_{B_{R}(0)}|Q_{n}|dx\rightarrow 0, \quad \forall R>0.
\end{equation} 
On the other hand, for each $\tau>0$,  we can fix $R$ large enough a such way that 
$$
\int_{B_{R}(0)^{c}}|f(v)v^{+}|dx<\tau.
$$ 
\begin{claim} \label{intf} Increasing $R$ if necessary, we also have 
$$
\int_{B_{R}(0)^{c}}|f(v_{n})v_{n}^{+}|dx<2\Theta\tau, \quad \forall n \in \mathbb{N}
$$
where
$$
\Theta:=\sup_{n\in\N}\left\{\left(\int|v_{n}^{+}|^{q+1}dx\right)^{\frac{1}{q+1}},\int|v_{n}v_{n}^{+}|dx\right\}.
$$
\end{claim}
In fact, for each $\tau>0$, the Lemma \ref{geje} ensures the existence of $c_{\tau}>0$ such that 
$$
|j_{\tau}(t)|^{r}\leq c_{\tau}tf(t), \quad \mbox{where} \quad r=\frac{q+1}{q}. 
$$
From Lemma \ref{geje}, 
$$\int_{B_{R}(0)^{c}}|f(v_{n})v_{n}^{+}|dx=\int_{B_{R}(0)^{c}}|g_{\tau}(v_{n})||v_{n}^{+}|dx+\int_{B_{R}(0)^{c}}|j_{\tau}(v_{n})||v_{n}^{+}|dx\leq$$
$$\leq\tau\int_{B_{R}(0)^{c}}|v_{n}||v_{n}^{+}|dx+\left(\int_{B_{R}(0)^{c}}|j_{\tau}(v_{n})|^{r}dx\right)^{1/r}\left(\int_{B_{R}(0)^{c}}|v_{n}^{+}|^{q+1}dx\right)^{1/(q+1)}$$
$$
\leq\tau\Theta+\left(\int_{B_{R}(0)^{c}}c_\tau f(v_{n})v_{n}dx\right)^{1/r}\Theta\leq\tau\Theta+c_\tau\left(\int_{B_{R}(0)^{c}}cH\ dx\right)^{1/r}\Theta.
$$
Now, increasing $R$ if necessary, a such way that  
$$
c_\tau\left(\int_{B_{R}(0)^{c}}cH\ dx\right)^{1/r} < \tau
$$
we get
$$
\int_{B_{R}(0)^{c}}|f(v_{n})v_{n}^{+}|dx\leq 2\tau\Theta,
$$
proving the claim. From (\ref{g}) and Claim \ref{intf},
$$
\int |Q_n|\,dx \to 0.  
$$
Therefore
$$
f(v_{n})v_{n}^{+}\rightarrow f(v)v^{+}\text{ in }L^{1}(\R^{N}).
$$
Analogously,
$$
f(v_{n})v_{n}^{-}\rightarrow f(v)v^{-}\text{ in }L^{1}(\R^{N}).
$$
Since $I_{n}'(u_{n})u_{n}^{+}=0$, it follows that
$$
||v_{n}^{+}||^{2}=\int A(\epsilon_{n}x+\epsilon_{n}y_{n})f(v_{n})v_{n}^{+}dx\rightarrow\int A(z)f(v)v^{+}dx=||v^{+}||^{2},
$$
showing that $v_{n}^{+}\rightarrow v^{+}$ in $H^{1}(\R^N)$, because $v_{n}^{+}\rightharpoonup v^{+}$ in $H^{1}(\R^N)$. Likewise $v_{n}^{-}\rightarrow v^{-}$ in $H^{1}(\R^N)$. Thereby $v_{n}=v_{n}^{+}+v_{n}^{-}\rightarrow v^{+}+v^{-}=v$ in $H^{1}(\R^N)$, finishing the proof.
\end{proof}
\begin{corollary}\label{uniformementeparazero}
$||v_{n}||_{L^{\infty}(\R^{N})}\not\rightarrow0$.
\end{corollary}
\begin{proof}
If $||v_{n}||_{L^{\infty}(\R^{N})}\rightarrow0$, by Proposition \ref{convergenciaforte}, we must have  $v=0$, which is absurd.
\end{proof}
\begin{lemma} \label{Regularidadeinfinito}
For all $n\in\N$, $v_{n}\in C(\R^{N})$. Furthermore, there exist a continuous function $P:\R\rightarrow\R$  with $P(0)=0$ and $K>0$ such that
$$
||v_{n}||_{C(\overline{B_{1}(z)})}\leq K\cdot P\left(||v_{n}||_{L^{2^{*}}(B_{2}(z))}\right),
$$
for all $n\in\N$ and for all $z\in\R^{N}$.
\end{lemma}

\begin{proof}
Since $u_{n}$ is solution of $(P)_{\epsilon_{n}}$, $v_{n}$ is a solution of
$$
\left\{
\begin{array}{l} -\Delta v_{n}+V(x)v_{n}=A(\epsilon_{n}x+\epsilon_{n}y_{n})f(v_{n})\quad \mbox{in} \quad \R^{N}, \\ v_{n}\in H^{1}(\R^{N}).
\end{array}
\right.
$$
Setting  $\Psi_{n}(x,t):=A(\epsilon_{n}x+\epsilon_{n}y_{n})f(t)$, it is easy to check that there exists $C>0$, independently of $n\in\N$, verifying 
$$
\Psi_{n}(x,t)\leq C(|t|+|t|^{q}), \quad \forall x \in \R^N \quad \mbox{and} \quad \forall t \in \R.
$$
Moreover, for each $z \in \R^N$ we have  that $u\in L^{s}(B_{2}(z))$ with $s\geq q$, $\Psi_{n}(\cdot, u(\cdot))\in L^{s/q}(B_{2}(z))$ and there exist $C_s=C(s)>0$, independent of $z$, such that 
$$
||\Psi_{n}(\cdot,u(\cdot))||_{L^{s/q}(B_{2}(z))}\leq C_s(||u||_{L^{s/q}(B_{2}(z))}+||u||^{q}_{L^{s}(B_{2}(z))}), \quad \forall n \in \mathbb{N}. 
$$
Here we have used the fact that $A$ is a bounded function. Now, recalling that potential $V$ is also a bounded function, we can proceed in the same manner as in  \cite[Proposition 2.15]{Rab} to get the desired result. 
\end{proof}

\vspace{0.5 cm}

As a byproduct of the last lemma we have the corollary below
\begin{corollary}\label{uniformementec0}
Given $\delta>0$, there exists $R:=R_{\delta}>0$ such that $|v_{n}(x)|\leq \delta$ for all $x\in\R^{N}\setminus B_{R}(0)$, that is, $\lim_{|x| \to +\infty}v_n(x)=0$ uniformly in $\mathbb{N}$. 
\end{corollary}
\begin{proof} Since $v_n \to v$ in $H^{1}(\R^N)$, given $\tau>0$ there are $R>0$ such that
$$
\|v_n\|_{L^{2^{*}}(B_{2}(z))} < \tau, \quad \mbox{for all} \,\,\, |z| \geq R \quad \mbox{and} \quad  n \in \mathbb{N}.
$$	
As $P$ is a continuous function and $P(0)=0$, given $\beta>0$, there is $\tau>0$ such that 
$$
|P(t)|< \beta / K, \quad \mbox{for} \quad |t| < \tau.
$$
Hence, by Lemma \ref{Regularidadeinfinito},
$$
||v_{n}||_{C(\overline{B_{1}(z)})}< \beta \quad \mbox{for} \quad |z| \geq R \quad \mbox{and} \quad  n \in \mathbb{N}.
$$
This proves the corollary. \end{proof}

\vspace{0.5 cm}

Finally we are ready to show the concentration of maxima. \\

\noindent {\bf Concentration of maxima:} \\

From Corollary \ref{uniformementec0}, there is $z_{n}\in\R^{N}$ such that $|v_{n}(z_{n})|=\max_{x \in \R^{N}} |v_{n}(x)|$. Now, applying Corollary \ref{uniformementeparazero}, there exists $\delta>0$ such that $|v_{n}(z_{n})|\geq\delta$ for all $n\in\N$, implying that $(z_{n})$ is bounded. Therefore if $\xi_{n}:=z_{n}+y_{n}$, it follows that 
$$
|u_{n}(\xi_{n})|=\max_{x \in \R^{N}}|u_{n}(x)|
$$
and 
$$
\epsilon_{n}\xi_{n}=\epsilon_{n}z_{n}+\epsilon_{n}y_{n}\rightarrow 0+z=z
$$
with $z \in \mathcal{A}$, finishing the study of the concentration phenomena.

\vspace{0.5cm}
\noindent \textbf{Acknowledgment}: The authors warmly thank the anonymous referee for his/her useful and nice comments on the paper.

\end{document}